\documentclass{amsart}
\usepackage[utf8]{inputenc}
\usepackage{amssymb,amsfonts, stmaryrd, amsmath, amsthm, enumerate, mathtools, hyperref}
\usepackage{wrapfig,tikz, tikz-cd}
\usetikzlibrary{arrows}
\usepackage{color}

\newtheorem{theorem}{Theorem}
\newtheorem{lemma}{Lemma}
\newtheorem{proposition}{Proposition}
\newtheorem{corollary}{Corollary}

\theoremstyle{definition}
\newtheorem{example}{Example}
\newtheorem{remark}{Remark}

\newcommand{\Hom}{{\operatorname{Hom}}}
\newcommand{\sHom}{\underline{\operatorname{Hom}}}
\newcommand{\HH}{{\operatorname{HH}^*}}
\newcommand{\HHF}{{\operatorname{HH}^1}}
\newcommand{\sHH}{{\underline{\operatorname{HH}}^*}}
\newcommand{\Ext}{{\operatorname{Ext}}}
\newcommand{\sExt}{\underline{\operatorname{Ext}}}
\newcommand{\RHom}{{\operatorname{RHom}}}

\newcommand{\sing}{\underline{\operatorname{D}}^{\sf b}}
\newcommand{\dcat}{{\operatorname{D}}}
\newcommand{\dbcat}{{\operatorname{D}^{\sf b}}}

\newcommand{\csing}{\underline{\operatorname{D}}^{\sf B}}
\newcommand{\cdbcat}{\operatorname{D}^{\sf B}}

\renewcommand{\flat}{\operatorname{Flat}}
\newcommand{\idim}{\operatorname{injdim}}

\renewcommand{\O}{\mathcal{O}}
\newcommand{\B}{\mathcal{B}}
\newcommand{\coder}{\operatorname{coDer}}

\newcommand{\susp}{{\mathtt{\Sigma}}}

\renewcommand{\-}{\text{-}}
\newcommand{\op}{\textsf{op}}
\newcommand{\ev}{\textsf{ev}}

\title[Stable Invariance of Restricted Structure]{Stable Invariance of the Restricted Lie Algebra Structure of Hochschild Cohomology}

\author{Benjamin Briggs}
\author{Lleonard Rubio y Degrassi}


\begin{document}

\begin{abstract}
 We show that the restricted Lie algebra structure on Hochschild cohomology is invariant under stable equivalences of Morita type between self-injective algebras. Thereby we obtain a number of positive characteristic stable invariants, such as the $p$-toral rank of $\mathrm{HH}^1(A,A)$. 
    We also prove a more general result concerning Iwanaga-Gorenstein algebras, using a more general notion of stable equivalences of Morita type. Several applications are given to commutative algebra and modular representation theory. 
    
    These results are proven by first establishing the stable invariance of the $B_\infty$-structure of the Hochschild cochain complex. In the appendix we explain how the  $p$-power operation on Hochschild cohomology can be seen as an artifact of this $B_\infty$-structure. In particular, we establish well-definedness of the $p$-power operation, following  some---originally topological---methods due to May, Cohen and Turchin, using the language of operads. 
\end{abstract}

\maketitle
It is a central theme in representation theory to understand what structure possessed by a finite dimensional associative algebra is shared by algebras with equivalent module categories, or derived categories, or stable module categories. 
Stable equivalences  
are both more frequent 
and less understood than derived equivalences, 
and for this reason it is desirable to have as many invariants under stable equivalence as possible.

Hochschild cohomology possesses a great deal of structure, and this is an important sourse of invariants. Perhaps most famously, it is a Gerstenhaber algebra, and this is enriched by a $B_{\infty}$-algebra structure on the Hoschschild cochain complex. For more specific classes of algebras, more structure appears. 
The invariance of these structures has been studied by various authors. 
For derived equivalences, Keller 
showed the invariance of the Gerstanhaber structure \cite{KellerIII}, 
and the invariance of $B_{\infty}$-structure of the
cochain complex \cite{KellerI}.

To better understand stable equivalences, Brou\'e introduced stable equivalences of Morita type in \cite{B}, 
and Keller and Vossieck \cite{KV} and Rickard \cite{Ri} proved that derived equivalences give rise to stable equivalences of Morita type for self-injective algebras. Given such an equivalence connecting two  self-injective 
algebras $A$ and $B$, Xi showed that the
Hochschild cohomology groups $\mathrm{HH}^n(A,A)$ and $\mathrm{HH}^n(B,B)$ are 
isomorphic  for $n\geq 1$ \cite{X}. In a similar vein, K\"onig, Liu and Zhou proved established  the invariance of the Batalin–Vilkovisky structure on stable Hochschild cohomology for symmetric algebras \cite{KLZ}.

 Over a field of positive characteristic, Hochschild cohomology 
 possesses also the structure of a  restricted graded Lie 
 algebra. 
In this context Linckelmann asked the following question: 
\emph{is the restricted Lie algebra structure of $ {\rm HH}^{>0}(A,A)$ invariant under stable equivalences of Morita type for self-injective algebras?} 
The invariance of this structure is delicate because the $p$-power structure is \emph{non-linear}, and 
therefore difficult to handle functorially. A partial answer was given by the second author in \cite{RyD} for the subclass of integrable derivations. We answer Linckelmann's question affirmatively:

\begin{theorem}
\label{stableinvariancetheorem} 
Let $A$ and $B$ be finite dimensional, self-injective algebras over a field of positive characteristic. If $A$ and $B$ are stably equivalent of Morita type, then the induced transfer map on Hochschild cohomology gives an isomorphism $ {\rm HH}^{>0}(A,A)\cong {\rm HH}^{>0}(B,B)$ of  \underline{restricted} graded Lie algebras.
\end{theorem}

Our main result applies to a pair of Iwanaga Gorenstein algebras (see Theorem \ref{generalVersion}).
In fact, we prove a much more general result about about stable invariance of the $B_\infty$-structure on $C^*(A,A)$ by extending arguments of Keller.

Theorem \ref{stableinvariancetheorem} yields nontrivial information in the problem of distinguishing stable 
equivalence classes, because the same graded Lie algebra can admit different restricted 
structures (for example, for abelian Lie algebras see \cite[Chapter 3]{StraFar}).

\begin{corollary}\label{resinvscor}
The following are invariant under stable equivalences of Morita type between finite dimensional self-injective algebras:
\begin{itemize}
    \item the restricted Lie algebra ${\rm HH}^1(A,A)$;
    \item the restricted Lie algebra ${\sf Z}({\rm HH}^1(A,A))$;
    \item the restricted universal enveloping algebras of ${\rm HH}^1(A,A)$ and of ${\sf Z}({\rm HH}^1(A,A))$;
    \item the maximal dimension of a $p$-toral subalgebra of ${\rm HH}^1(A,A)$.
    
\end{itemize}
\end{corollary}

Corollary \ref{resinvscor} yields important information because the Lie algebra  $\HHF(A)$ is known to contain significant information 
about $A$; see \cite{LiRu1,Chang,LiRu2,ER,RSS,BKL2}. In \cite{BR} the present authors characterise the $p$-toral subalgebras $\HHF(A,A)$ in terms of the fundamental groups associated to presentations in $A$;  Corollary \ref{resinvscor} is then critical in using this to construct new invariants under stable equivalences of Morita type.

As well as the $p$-power operation, there is another invariant defined on the even Hochschild cohomology classes which does not seem to have been used in representation theory.  This operation is analogous to the Cohen-Dyer-Lashof operations from topology. We show that the extra structure is also invariant under stable equivalences of Morita type---see Remark \ref{CDLoperation}.

We provide some applications to modular representation 
theory and commutative algebra. In section \ref{tamesym} we consider the classification of algebras of dihedral, semi-dihedral and quaternion type up to stable equivalence of Morita type, as started by Zhou, Zimmermann and  Taillefer in \cite{ZZ,ZI,Taillefer} who employed a number of   invariants, 
 including the Lie structure of $\mathrm{HH}^1$,   to distinguish most of the algebras. We compute 
the restricted Lie structure of most of the symmetric algebras of semi-dihedral and quaternion type that have not been classified. The restricted structure
fails to distinguish classes, providing some evidence that in the remaining cases derived  and stable classification should  coincide. The second application is based on the work of 
\cite{BKL}; as a consequence of Theorem \ref{stableinvariancetheorem} we obtain an 
isomorphism as restricted Lie algebras of $\mathrm{HH}^1$ for all non-nilpotent blocks with defect two and 
their Brauer correspondents. Lastly, in Theorem \ref{commutativetheorem} we apply our results to
commutative Gorenstein rings. In particular we show that Kn\"orrer periodicity induces an isomorphism of restricted Lie algebras in Hochschild cohomology.

One of the main insights is to view $p$-power structure on Hochschild cohomology as an artifact of the $B_\infty$-structure on the Hochschild cochain complex (this idea arises in topology \cite{Cohen,Tourtchine}). The usefulness of this perspective is evidenced by the functoriality in Theorem \ref{E2restrictedLie}. This is related to a conjecture of Keller that the $B_\infty$-structure on the Hochschild cohomology of $A$ can be obtained directly from the enhanced singularity category of $A$; see \cite[1.2]{KellerII} for a precise statement.

Even though this work is aimed at modular representation theorists, the proof of Theorem \ref{stableinvariancetheorem} consists of combines concepts from a number of different areas, and we have tried to include all of the needed background.

To show that the $p$-power operation is well-defined it seems necessary to involve some operadic machinery. We do this in the appendix, following work of May and Cohen in topology. The results therein serve as general template for building cohomology operations, and so we hope this appendix might be a useful reference, since these arguments do not seem to have been written down before in their algebraic form.

\tableofcontents

\noindent
{\bf Conventions.} Throughout $k$ is a commutative noetherian ring, and everything is taken to be $k$-linear. Undecorated tensor products 
are taken over $k$, and many of the objects we will discuss (notably bimodules and Hochschild cohomology) depend implicitly 
on $k$. Not much is lost in imagining $k$ to be a field. We use the notation $(\susp M)_i= M_{i-1}$ for the  shift of a graded object $M$.

\noindent
{\bf Acknowledgements.}
We are happy to thank Victor Turchin for explaining to us some facts about cohomology operations (see the appendix). We are thankful as well to Bernhard Keller for several helpful correspondences, and for pointing out an inaccurate statement about the existence of stable restriction functors (see Section \ref{MoreFunctorialitySubsec}). Lastly, we thank Markus Linckelmann for his support, his comments and for posing the question about the invariance of the $p$-power map.

\medskip

\noindent
{\bf Funding.}
The second author has been  supported by the Fundaci\'on `S\'eneca' of Murcia
(19880/GERM/15) and by the INdAM postdoctoral research grant 2019-2020.

\section{The Structure and Functoriality of Hochschild Cohomology}\label{structureofHHSection}

We will borrow an argument from Keller \cite{KellerI} and exploit the functoriality of Hochschild cohomology 
for fully faithful embeddings of 
categories. This is useful even if we start with  algebras, but since it adds 
no extra difficulty we may as well work with categories (or ``algebras with many objects" \cite{Mitchell}) from the beginning.

If $A$ is a small $k$-linear 
category then we use the notation ${}_b A_a=\Hom_A(a,b)$ for brevity. 
A right dg $A$-module $M$ consists of a complex $M_a$ of  $k$-modules for each object $a$ and an action 
$M_b\otimes {}_bA_a\to M_a$ for objects $a$ and $b$ of $A$, associative and unital in the natural sense. Left modules, and bimodules, 
are defined similarly. In particular, there is for each object $a$ the (representable) right $A$-module ${}_aA$ and 
the (corepresentable) left $A$-module $A_a$. These play the role of free modules.  In addition, $A$ is  naturally a 
bimodule over itself. See \cite{Toen} for more details on modules over (dg) 
categories and their derived categories.

Always $\dcat(A)$ denotes the full derived category of $A$, and $\dcat(A\- B)$ is the derived category of 
$A\- B$ dg bimodules (with $k$ acting centrally). In particular $\dcat(A\- A)\simeq \dcat(A^\ev)$ where 
$A^\ev=A^\op\otimes A$.

The Hochschild cohomology of $A$ is $\HH(A,A)= {\rm End}_{\dcat(A\- A)}(A)=\Ext_{A\- A}(A,A)$. Note that this presents  
$\HH(A,A)$ as a graded $k$-algebra.

\subsection{B\texorpdfstring{$_\infty$}{}-structure} Hochschild cohomology enjoys a great deal more structure  than its associative product alone, 
and this structure manifests most clearly at the chain level. Thus, going back to \cite{Mitchell} 
one defines the Hochschild cochain complex of the category $A$ to be
\[
C^n(A,A)=\Hom(A^{\otimes n},A)=\prod_{a_n,...,a_0}\Hom({}_{a_n}A_{a_{n-1}}\otimes \cdots \otimes {}_{a_1}A_{a_{0}},{}_{a_{n}}A_{a_0}),
\]
where $a_i$ for $0\leq i\leq n$ are objects of $A$. The Hochschild differential has its classical formula. As long as $A$ is projective over $k$, the cohomology of $C^*(A,A)$ canonically computes $\HH(A,A)$.

The cup product in this context is also given by the classical formula: if $x\in \Hom({}_{a_n}A_{a_{n-1}}\otimes \cdots \otimes {}_{a_1}A_{a_{0}},{}_{a_{n}}A_{a_0})$ and $y\in\Hom({}_{b_m}A_{b_{m-1}}\otimes \cdots \otimes {}_{b_1}A_{b_{0}},{}_{b_{m}}A_{b_0})$, with $a_0=b_m$, then $
x\smile y $ is the linear map
\[
{}_{a_n}A_{a_{n-1}}\otimes \cdots \otimes {}_{b_1}A_{b_{0}} \xrightarrow{x\otimes y} {}_{a_{n}}A_{a_0}\otimes{}_{b_{m}}A_{b_0}\xrightarrow{\rm comp.}  {}_{a_{n}}A_{b_0}
\]
(the product is zero if $a_0\neq b_m$). Altogether this makes $C^*(A,A)$ into a dg algebra.

The next series of operations may at first seem mysterious, but all we really need from them are the following two facts: firstly, the \emph{restricted} Lie algebra structure on $\HH(A,A)$ is naturally built from them; and secondly, the operations are preserved by the restriction maps  discussed in the next subsection.

The \emph{brace operations} on $C=C^*(A,A)$ are, for each $s\geq 1$, linear maps
\[
(\-)\{\-,\dots,\-\} \colon C\otimes C^{\otimes s}\to C
\]
of degree $-s$. Thus, if $x\in C^n(A,A)$ and $y_i\in C^{n_i}(A,A)$ then
\[
x\{y_1,\dots,y_s\} \in C^{n+n_1+\cdots+ n_s-s}(A,A).
\]
To define these operations, we consider all ways of applying the $y_i$ in parallel: if $m_0+n_1+m_1+\cdots+ m_{s-1}+n_s+m_s= n+n_1+\cdots +n_s-s$, then denote
\[
y^{(\underline{m},\underline{n})}\colon A^{\otimes(n+n_1+\cdots+ n_s-s)}\xrightarrow{1^{\otimes m_0}\otimes y_1 \otimes 1^{\otimes m_1} \cdots 1^{\otimes m_{s-1}}\otimes y_s \otimes 1^{\otimes m_s} } A^{\otimes n}.
\]
Then we take the signed sum of these plugged into $x$, so
\[
x\{y_1,\dots,y_n\} = \sum_{\underline{m}}(-1)^{(m_0+n_1+\cdots +m_{s-1})(n_s+1) + \cdots + (m_0)(n_1+1)} y^{(\underline{m},\underline{n})} x.
\]
(We have suppressed the objects of $A$ here for readability---as in the formula for the cup products we need all 
the objects to ``match up'' to get a nonzero component.)

For example, if ${\bf m}\colon A^{\otimes 2}\to A$ is the multiplication (or composition) on $A$, then the 
Hochschild differential is given by $\partial(x)={\bf m}\{x\}+(-1)^{|x|}x\{{\bf m}\}$, 
and the cup product is given by $x\smile y ={\bf m}\{x,y\}$. 

The binary operation $(\-)\{\-\}$ coincides with Gerstenhaber's \emph{circle product} $x\{y\} = x\circ y$. This is the only brace operation we need in order to discuss the restricted 
Lie algebra structure of $\HH(A,A)$.

These brace operations were introduced by Getzler in \cite{Getzler}. The various compatibility conditions between them can be summed up by saying that they induce a product on the \emph{bar construction} $BC=\bigoplus_i (\susp C)^{\otimes i}$ which makes $BC$ into a dg Hopf algebra. By definition, this makes $C^*(A,A)$ an example of a  \emph{$B_{\infty}$-algebra}. See  Getzler and Jones \cite[Section 5.2]{GJ} or Gerstenhaber and Voronov \cite[section 3.2]{GV} for details. 

As usual, much of this was anticipated by Gerstenhaber long ago: what is roughly the same structure on the Hochschild chain 
complex was used in \cite{Gerst} under the name ``composition complex". The cochain algebra of a topological space, and the 
cobar construction of a Hopf algebra, were also given as examples of composition complexes. These would both now be 
considered standard examples of $B_{\infty}$-algebras. 

It has been proven by McClure and Smith that a $B_\infty$-algebra is essentially the same thing as an $E_2$-algebra \cite{McClureSmith}, answering Deligne's question about the  structure of $C^*(A,A)$.

\subsection{Functoriality}\label{FunctorialitySubsec}
Let $A$ and $B$ be categories which are flat over $k$ (i.e.~all of the hom-spaces are flat as $k$-modules) and let $\phi:A\to B$ is a fully faithful embedding. Then there is a restriction functor 
$R_\phi: \dcat(B)\to \dcat(A)$ which sends a $B$-module $M$ to the $A$-module with components $(R_\phi M)_a=M_{\phi(a)}$. 
This has a left adjoint $L_\phi:\dcat(A)\to \dcat(B)$ given by $L_\phi(M)= M\otimes^{\sf L}_BA$. Since 
$R_\phi L_\phi\simeq {\rm id}_{\dcat(A)}$ the functor $L_\phi$ is fully-faithful.

From $\phi $ we get a fully faithful embedding $\phi^\ev :A^\ev\to B^\ev$ and hence a restriction 
$R_{\phi^\ev}:\dcat(B\- B)\to \dcat(A\- A)$. Since $R_{\phi^\ev}(B)=A$, this induces a homomorphism 
${\rm End}_{\dcat(B\- B)}(B)\to {\rm End}_{\dcat(A\- A)}(A)$ of graded rings. 
Identifying these rings with the Hochschild cohomology of $A$ and $B$ respectively we have a homomorphism
\[
\phi^*: \HH(B,B)\to \HH(A,A).
\]
This homomorphism has been exploited by Keller, who observed in \cite{KellerI} that one can lift $\phi^*$ to 
the Hochschild cochain complex in such a way that it preserves the $B_\infty$-structure.

Indeed, we define $
\phi^*: C^*(B,B)\to C^*(A,A)$ by its components
\[
\Hom({}_{\phi(a_n)}B_{\phi(a_{n-1})}\otimes \cdots \otimes {}_{\phi(a_1)}B_{\phi(a_{0})},{}_{\phi(a_{n})}B_{\phi(a_0)})\quad\quad\quad\quad\quad\quad\quad
\]
\[\quad\quad\quad\quad\quad\quad\quad\xrightarrow{\ \cong\ } \Hom({}_{a_n}A_{a_{n-1}}\otimes \cdots \otimes {}_{a_1}A_{a_{0}},{}_{a_{n}}A_{a_0}).
\]
That is, $\phi^*$ simply restricts each $k$-linear map in $C^*(B,B)$ to those components only involving objects in the image of $A$. One can check that this is a chain map which induces the homomorphism $\phi^*: \HH(B,B)\to \HH(A,A)$ described above.

It is evident from the formulas in the previous subsection that $
\phi^*: C^*(B,B)\to C^*(A,A)$ respects all of the brace operation. This observation of Keller's is extremely useful. It follows that all operations on  Hochschild cohomology which can be written in terms of the brace structure   
must be preserved by $\phi^*$.

\subsection{Graded Lie Algebras}

A graded Lie algebra is a graded $k$-module $L$  equipped with a bilinear bracket 
$[\- ,\- ]: L\times L\to L$ and with a quadratic \emph{reduced square} operation 
defined on odd elements $(\- )^{[2]}: L^{2i+1}\to L^{4i+2}$,  so $(\alpha 
u)^{[2]}=\alpha^2 u^{[2]}$ for any  $\alpha$ in $k$. These operations should 
satisfy various axioms (compatibility, anti-symmetry and the graded Jacobi 
identity) which can be found in \cite[Remark 10.1.2]{Avramov}. 
 The reduced square is redundant if $2$ is a unit in $k$.

\begin{example}
\label{ex1}Any graded associative algebra $A$ over $k$ gives rise to a graded Lie algebra 
$A^{\sf Lie}$ with the same underlying graded vector space. The bracket is given 
by the graded commutator $[x,y]= xy-(-1)^{|x||y|}yx$ while the reduced square is 
simply the square $u^{[2]}=u^2$. 
 If $A$ has the structure of a graded Hopf algebra, then the space of primitives in
 $A$ is a sub-graded Lie algebra of $A^{\sf Lie}$.
 \end{example}

\begin{example}
\label{ex2}
The degree-shifted Hochschild cochain complex $\susp C^*(A,A)$ is another example. The classical Gerstenhaber bracket is $[x,y]= x\circ y-(-1)^{|x||y|}y\circ x$ 
and the reduced square is $x^{[2]}=x\circ x$. 
These operations descend to cohomology, making  $\susp\HH(A,A)$ into a graded Lie algebra.

In particular, ${\rm HH}^1(A,A)$ is an (ungraded) Lie algebra. This is an important---and often computable---invariant in representation theory.
\end{example}

\subsection{Restricted Graded Lie Algebras}\label{restrictedsubsection} So far none of this depends on the characteristic of $k$. 
The next structure belongs to the positive characteristic world. So, from now on $p$ is a fixed, positive prime, and we assume that $k$ contains a field of characteristic $p$.

A \emph{restricted} graded Lie algebra is a graded Lie algebra $L$ equipped with a $p$-power-linear operation defined on even elements $(\- )^{[p]}:L^{2i}\to L^{2ip}$, so $(\alpha u)^{[p]}=\alpha^p u^{[p]}$ for any $\alpha$ in $k$. This restricted structure must satisfy two standard compatibility conditions with the Lie structure, which can be found in \cite[section 4.2]{ZII} or \cite[section 7]{Tourtchine}.

\begin{example}
If $A$ is a graded associative algebra, then $A^{\sf Lie}$ admits a restricted Lie algebra structure by defining $x^{[p]}=x^p$ on even elements. If, further, $A$ is a graded Hopf algebra, then $p$th powers of primitive, even degree  elements remain primitive because of the formula $(x\otimes 1+1\otimes x)^p = x^p\otimes 1+1\otimes x^p$. Thus the space of primitives in $A$ is also a restricted Lie algebra.
\end{example}

\begin{example}\label{poweropexample}
For us the most important  example is the shifted Hochschild cohomology $\susp {\rm HH}^{*}(A,A)$. On top of the Gerstenhaber bracket there is a $p$-power operation given on $C^*(A,A)$ by
\[
x^{[p]}= x\{x\}\cdots \{x\} = ((x\circ x) \circ \cdots )\circ x
\]
with $x$ appearing $p$ times in either expression, see  \cite[section 3]{Tourtchine}. It is difficult to show that this operation is well-defined on cohomology classes, a proof is given in the appendix.
\end{example}

More classically, this restricted Lie algebra structure is often presented by identifying $C^*(A,A)$ with the space of coderivations on the bar construction $BA$ 
\[
C^*(A,A)\cong {\rm coDer}(BA,BA).
\]
This isomorphism was discovered independently by Quillen \cite{Quillen} and Stasheff \cite{Stasheff}. The Gerstenhaber bracket under this identification becomes the graded commutator of coderivations, and the $p$-power structure is given by $p$-fold composition. See \cite{KellerIII} or \cite{ZII} for details. 

The following fact deserves to be more well-known in algebra. It is one reason that restricted Lie algebras often
occur in homotopical situations, and it gives a convenient way to handle functoriality. The (originally
topological) result is due to Cohen  \cite{Cohen}.

\begin{theorem}\label{E2restrictedLie}
If $C$ is a $B_\infty$-algebra then $\susp{\rm H_*}(C)$ is naturally a restricted graded Lie algebra. A map of
$B_\infty$-algebras induces a map of restricted graded Lie algebras in homology.

In particular, if $\phi: A\to B$ is a fully faithful embedding of dg categories then $\phi^*:\HH(B,B)\to \HH(A,A)$
is (after shifting) a homomorphism of restricted Lie algebras.
\end{theorem}

In other words, the homology of any $B_\infty$-algebra is a ``restricted Gerstenhaber algebra''. 

\begin{proof}
It is well-known that the cohomology of a $B_\infty$-algebra is a Gerstenhaber algebra \cite{GV}, so we only need to explain the $p$-power operation. This is given by the same formula as in example \ref{poweropexample}. The proof that this operation is well-defined turns out to be subtle, so it is dealt with in the appendix.

The statement about Hochschild cohomology follows from Keller's observation \cite{KellerI} that the restriction map $C^*(B,B)\to C^*(A,A)$ is one of $B_\infty$-algebras.
\end{proof}

\begin{remark}
\label{CDLoperation}
The $p$-power map is (when $p$ is odd) just one of the two Cohen-Dyer-Lashof operations which exists on
the homology of any $B_\infty$-algebra, see \cite{Cohen,Tourtchine}. The other operation is defined on
odd cycles by the formula 
 $\zeta(x)= \sum_{i+j=p-1}\frac{(-1)^i}{i}x^{[i]}\smile x^{[j]}$ as in \cite[theorem 3.1]{Tourtchine}. It
 doesn't seem that this structure has been exploited by algebraists. For example, using the methods of
 this paper the induced $\zeta$ operation on Hochschild cohomology could conceivably be used to
 distinguish algebras up to stable equivalence of Morita type.
\end{remark}

\begin{corollary}\label{restrictedderivedinv} 
If $A$ and $B$ are two dg algebras which are derived equivalent, then there is an isomorphism $\susp {\rm HH}^{>0}(A,A)\cong\susp {\rm HH}^{>0}(B,B)$ of restricted graded Lie algebras.
\end{corollary}

This was originally stated by Zimmerman \cite[proposition 4.4]{ZII}. 
The proof in \cite{ZII} seems to be incomplete, because some of the maps used (between complexes of coderivations) are not well-defined at the chain level.

Using Theorem \ref{E2restrictedLie}, Corollary \ref{restrictedderivedinv} now follows from a result of Keller \cite{KellerI}: if $A$ and $B$ are derived equivalent then $C^*(A,A)$ and $C^*(B,B)$ are connected by a zig-zag of quasi-isomorphisms of $B_\infty$-algebras. The main result of this paper is the generalisation to stable equivalences of Morita type.

As another application, the next corollary follows in the same way from Theorem \ref{E2restrictedLie} and \cite[theorem 3.5]{KellerI}.

\begin{corollary}\label{KDcor} 
If $A$ and $B$ are a Koszul dual pair of Koszul algebras, then 
$\susp {\rm HH}^{>0}(A,A)\cong\susp {\rm HH}^{>0}(B,B)$ as restricted graded Lie algebras.
\end{corollary}

The invariance of the $B_\infty$-structure under Koszul duality has been extended to not-necessarily Koszul algebras in \cite{KellerIV} (but here one uses Koszul–Moore duality, producing a dg coalgebra as the Koszul dual object). In this context one also has a $p$-power operation, and the analogue of Corollary \ref{KDcor} holds as well.

\subsection{Singularity Categories and Stable Hochschild Cohomology}\label{stableHHSection}

Let $\dbcat(A)$ be the homotopy category of bounded below complexes of finitely generated projectives that have bounded homology. The singularity category $\sing(A)$ of an algebra (or category) $A$ 
is defined as the Verdier quotient of  $\dbcat(A)$ by its subcategory ${\rm perf}(A)$ of perfect complexes, that is, bounded complexes of projective modules (cf.~\cite{Buc} or \cite{Orl}). 

Following Buchweitz \cite{Buc}, the \emph{stable} (or \emph{singular}, or \emph{Tate}) Hochschild cohomology of $A$ is the graded endomorphism algebra
\[
\sHH(A,A)=\sExt_{A\- A}(A,A)={\rm End}_{\sing(A^\ev)}(A).
\]
For symmetric algebras the positive part of $\sHH(A,A)$
coincides with $\HH(A,A)$ (Corollary 6.4.1 in \cite{Buc}).

This definition only works if $A$ is an object of $\dbcat(A^\ev)$. This is automatic if $A$ is finite over $k$, but in some contexts this is a strong condition on $A$. To avoid worrying about this we can use a large version of the singularity category  from  \cite{Orl}, as Keller has done in \cite{KellerII}.

Let $\cdbcat(A)$ denote the full subcategory of $\dcat(A)$ consisting of dg modules with bounded homology, 
and let $\flat(A)$ denote the thick subcategory of $\dcat(A)$ consisting of bounded complexes of flat modules. By definition $\csing(A)$ is the Verdier quotient $\cdbcat(A)/\flat(A)$. With this large version of the singularity category we may make the general definition
\[
\sHH(A,A)=\sExt_{A\- A}(A,A)={\rm End}_{\csing(A^\ev)}(A).
\]
Through the projection functor $\cdbcat(A\- A)\to\csing(A\-A)$ we have a natural stabilization homomorphism $\HH(A,A)\to\sHH(A,A)$. We must justify that this is consistent with our earlier definition:
\begin{proposition}[{\cite[proposition 1.13]{Orl}} or {\cite[lemma 2.2]{KellerII}}]
The natural functor 
$
\sing(A)\longrightarrow\csing(A)
$ 
is fully-faithful.
\end{proposition}
The statement here is slightly stronger than that of \cite{Orl} or \cite{KellerII}, but the same proof works (keeping in mind that flat modules are the filtered colimit of their projective submodules).

\subsection{More On Functoriality}\label{MoreFunctorialitySubsec} If  $\phi:A\to B$ is a fully faithful embedding, then $R_\phi:\dcat(B)\to \dcat(A)$ always preserves the bounded derived categories. If further each of the right $A$-modules  $R_\phi ({}_{b}B)$ have finite flat dimension, then this functor descends to a stable restriction functor $\underline{R}_\phi:\csing(B)\to \csing(A)$.

On the other hand, $L_\phi:\dcat(A)\to \dcat(B)$ restricts always to perfect objects, and to $\flat(A)$. If we assume that each of the left $A$-modules $R_\phi( B_b)$ have finite flat dimension, then $L_\phi$ will restrict to bounded derived categories, and therefore induce a stable extension of scalars $\underline{L}_\phi:\csing(A)\to\csing(B)$.

When $\phi$ has finite flat dimension on both sides (i.e. all of the restrictions ${}_bB$ and $B_b$ have finite flat dimension), both functors $\underline{L}_\phi$ and $\underline{R}_\phi$ exist and are left and right adjoint to each other respectively
\[
\begin{tikzcd}
 \underline{L}_\phi:\csing(A)\ar[r, shift left = 1mm ]& \ar[l, shift left = 1mm ]\csing(B):\underline{R}_\phi
\end{tikzcd}
\]
If we had assumed that the restrictions ${}_bB$ and $B_b$ were all \emph{perfect}, we would have obtained an adjunction between the small singularity categories $\sing(A)$ and $\sing(B)$. Even in good situations when the small singularity categories are available, the existence of $\underline{L}_\phi$ and $\underline{R}_\phi$ in this greater generality makes $\csing(A)$ useful.

\section{Stable equivalence of Morita type}\label{SEMTsection}

The setup in this section is that we have two flat 
$k$-algebras (or categories) $A$ and $B$, and an $A\- B$ bimodule $M$ which has finite flat dimension on both sides (i.e. as a left $A$ and a right $B$-module separately). 
We'll always use $a$ for an arbitrary object of $A$, and $b$ for an arbitrary object of $B$.

We borrow a construction from Keller \cite{KellerI}. 
Let $T$ the category whose objects are the disjoint union of those from $A$ and $B$, and whose morphisms are given by ${}_aT_a={}_aA_a$, ${}_bT_b={}_bB_b$, ${}_aT_b={}_aM_b$ and ${}_bT_a=0$. 
Composition in $T$ is given by composition in $A$ and $B$ and by the action of $A$ and $B$ on $M$. 
Diagramatically, $T$ looks like
\[
\begin{tikzcd}[column sep = 15mm]
    a \ar[loop, looseness=5,"A"'] & \ar[l,"M"'] b \ar[loop , looseness=4,"B"'] 
\end{tikzcd}
\]

and $T$ comes with two fully faithful embeddings $A\to T$ and $B\to T$.

We also have two fully faithful embeddings $A^\ev\to T^\ev$ and $B^\ev\to T^\ev$. 
The assumption that $M$ has finite flat dimension on both sides, and the assumption that $A$ and $B$ are flat over $k$, together imply that $T^\ev$ has finite flat dimension when restricted to either $A^\ev$ or $B^\ev$ (because the tensor product of two flat modules is flat). As explained in section \ref{MoreFunctorialitySubsec}, this implies that we have stable restriction functors  
\[
\underline{R}_A:\csing(A\- A)\to \csing(T\- T)\ \text{ and }\ \underline{R}_B:\csing(B\- B)\to \csing(T\- T).
\]

Since $R_A(T)=A$ and $R_B(T)=B$ we obtain a commutative diagram of homomorphisms
\begin{equation}\label{StabilisationDiagram}
\begin{tikzcd}
    \Ext_{A\- A}(A,A) \ar[d] & \Ext_{T\- T}(T,T) \ar[d] \ar[l,"R_A"'] \ar[r,"R_B"] & \Ext_{B\- B}(B,B) \ar[d]\\
    \sExt_{A\- A}(A,A) & \sExt_{T\- T}(T,T)  \ar[l,"\underline{R}_A"'] \ar[r,"\underline{R}_B"]& \sExt_{B\- B}(B,B).
\end{tikzcd}
\end{equation}

The key to the proof of Theorem \ref{stableinvariancetheorem} is to use this diagram with the following result. This is the technical core of the paper.

\begin{theorem}\label{keytheorem}
If $\-\otimes_A^{\sf L}M:\csing(A\- A)\to \csing(A\- B)$ and $M\otimes_B^{\sf L}\-:\csing(B\- B)\to \csing(A\- B)$ are both fully faithful then 
\[
\underline{R}_A:\sExt_{T\- T}(T,T)\to \sExt_{A\- A}(A,A)\quad\text{and}\quad
\underline{R}_B: \sExt_{T\- T}(T,T)\to\sExt_{B\- B}(B,B)
\]
are both isomorphims.
\end{theorem}

The hypotheses of this theorem are what we mean by stable equivalence of Morita type. Proposition \ref{SEMTprop} below connects this with the usual notion.

\begin{proof}
Associated to the embeddings of $A^\ev$, $B^\ev$ and $A^\op\otimes B$ into $T^\ev$ we have the left adjoints (from sections \ref{FunctorialitySubsec} and \ref{MoreFunctorialitySubsec}) which embed $\cdbcat(A\- A)$, $\cdbcat(B\- B)$ and $\cdbcat(A\- B)$ into $\cdbcat(T\- T)$, respectively. To simplify notation we denote all of these functors by $L$ (the domain should always be clear from context).

We start by showing that in $\cdbcat(T\- T)$ there is a natural triangle
\begin{equation}\label{basicTriangle}
   LM\to LA\oplus LB \to T\to \ .
\end{equation}

The adjunction $\Ext_{A\- A}(A,A)=\Ext_{A\- A}(A,RT)\cong \Ext_{T\- T}(LA,T)$ gives the map $LA\to T$ in this triangle. Similarly for $LB\to T$.

 We need to resolve $T$ in order to compute these left adjoints. By truncating a flat  $A^\op\otimes B$ resolution of $M$ somewhere above the flat dimensions of $M$ as a left or right $A$ or $B$-module, we obtain a quasi-isomorphism $\widetilde{M}\to M$ from a complex of bimodules which are all flat as left $A$ and right $B$-modules seperately (this uses flatness of $A$ and $B$ over $k$). 
 
There is then a natural $T$ bimodule (in fact, dg category) $\widetilde{T}$ given by
\[
\begin{tikzcd}[column sep = 15mm]
    a \ar[loop, looseness=5,"A"'] & \ar[l,"\widetilde{M}"'] b \ar[loop , looseness=4,"B"'] 
\end{tikzcd}
\]
and $\widetilde{T}$ admits a quasi-isomorphism to $T$. One can think of $\widetilde{T}$ as a flat left $A$-module resolution of $T$ (with a compatible right $T$ action). In turn, $\widetilde{T}\otimes\widetilde{T}$ is a flat left $A^\ev$-module resolution of $T\otimes T$  (with a compatible right $T^\ev$ action). This means $LA=A\otimes_{A^\ev}^{\sf L}(T\otimes T)\simeq A\otimes_{A^\ev}(\widetilde{T}\otimes\widetilde{T})$. This tensor product looks like
\[
\begin{tikzcd}[column sep = 15mm]
    a \ar[loop, looseness=5,"A"'] & \ar[l,"\widetilde{M}"'] b \ar[loop , looseness=4,"0"'] 
\end{tikzcd} \simeq
\begin{tikzcd}[column sep = 15mm]
    a \ar[loop, looseness=5,"A"'] & \ar[l,"{M}"'] b \ar[loop , looseness=4,"0"'] 
\end{tikzcd}
\]
i.e. ${}_a(LA)_a \simeq A$, ${}_a(LA)_b \simeq \widetilde{M}\simeq M$ and ${}_b(LA)_b=0$. The obvious inclusion into $T$ is the map in the triangle (\ref{basicTriangle}).

Similarly $LB$ is given by
\[
\begin{tikzcd}[column sep = 15mm]
    a \ar[loop, looseness=5,"0"'] & \ar[l,"{M}"'] b \ar[loop , looseness=4,"B"'] 
\end{tikzcd}
\]
and $LM$ is given by
\[
\!\!\!\!\begin{tikzcd}[column sep = 15mm]
    a \ar[loop, looseness=5,"0"'] & \ar[l,"{M}"'] b \ar[loop , looseness=4,"0"'] 
\end{tikzcd}\!\!\!\!\!\!\!.
\]
The map $LM\to LA\oplus LB$ is given by $M\to M\oplus M,\ m\mapsto (m,-m)$ in the ${}_a(LM)_b\to {}_a(LA\oplus LB)_b$ component, and (necesserily) zero at all other objects. This in fact produces a short exact sequence $0\to LM\to LA\oplus LB \to T\to 0$ of $T\-T$ bimodules, giving the desired triangle.

Since the components of $T$ are perfect as left and right modules over $A$ and $B$ respectively, all the functors named $L$ descend to functors $\underline{L}$ on singularity categories (this was explained in section \ref{stableHHSection}). Therefore, when we stabilize the triangle (\ref{basicTriangle}) by applying the quotient $\cdbcat(T\- T)\to \csing(T\- T)$, we get
\[
\underline{L}M\to \underline{L}A\oplus \underline{L}B \to T\to \ .
\]
Applying $\sExt_{T\- T}(-,T)$ produces the (Meyer-Vietoris) exact sequence
\[
\cdots\to \sExt_{T\- T}(T,T) \to \sExt_{T\- T}(\underline{L}A,T) \oplus
    \sExt_{T\- T}(\underline{L}B,T) \to \sExt_{T\- T}(\underline{L}M,T) \to \cdots
\]
By the adjunction between $\underline{L}$ and $\underline{R}$ (where $\underline{R}$ is the appropriate stabilized restriction map), and the fact that $\underline{R}\underline{L}$ is  naturally isomorphic to the identity, this long exact sequence is isomorphic to
\[
\cdots\to \sExt_{T\- T}(T,T) \to \sExt_{A\- A}(A,A) \oplus
    \sExt_{B\- B}(B,B) \to \sExt_{A\- B}(M,M) \to \cdots\ .
\]

In this sequence the two components of $\sExt_{T\- T}(T,T) \to \sExt_{A\- A}(A,A) \oplus
    \sExt_{B\- B}(B,B)$ are precisely the stabilized restrictions  $\underline{R}_A$ and $\underline{R}_B$ in the statement of the theorem.
    
To show that these are isomorphisms it is equivalent to show that the two components of $ \sExt_{A\- A}(A,A) \oplus
    \sExt_{B\- B}(B,B) \to \sExt_{A\- B}(M,M) $ are isomorphisms.  It is sufficient to show that they are each given by tensoring against $M$ (on the appropriate side), and therefore are isomorphisms by hypothesis. Before stabilizing, this means we must check that the following diagram commutes
\[
\begin{tikzcd}[column sep=15mm, row sep=5mm]
    \Ext_{T\-T}(LA,T) \ar[r]\ar[d,"\cong"] & \Ext_{T\-T}(LM,T) \ar[d,"\cong"]\\
    \Ext_{A\-A}(A,A) \ar[r,"-\otimes^{\sf L}_A M"] & \Ext_{A\-B}(M,M)
\end{tikzcd}
\]
One checks this by starting in the lower-left corner. The point is that the adjunction $\Ext_{A\-A}(A,A) \cong \Ext_{T\-T}(LA,T)$ by definition extends a map $\xi:A\to A$ to $LA\to T$  along the tensor product ${}_a(LA)_b=A\otimes^{\sf L}_{A^\ev}(A\otimes M) \simeq M \to M = {}_a T_b$. Since ${}_a(LA)_b\simeq {}_a(LM)_b$ the top arrow makes no difference after restricting to the lower-right corner. (One should really check this diagram commutes using projective bimodule resolutions of $A$ and $M$, but again this makes little difference.) After stablising the diagram the lower arrow becomes an isomorphism by hypothesis.

The verification for $B$ is very similar. This finishes the proof.
\end{proof}

Let us point out that the hypotheses of Theorem \ref{keytheorem} already imply 
that $M$ is one half of a stable equivalence of Morita type. This proposition is presumably well-known (at some level of generality).

\begin{proposition}\label{SEMTprop}
Let $M$ be an $A\- B$ dg bimodule 

which is perfect on both sides. If $\-\otimes_A^{\sf L}M:\sing(A\- A)\to \sing(A\- B)$ and $M\otimes_B^{\sf L}\-:\sing(B\- B)\to \sing(A\- B)$ are both fully faithful then there is a $B\- A$ dg bimodule $N$, perfect on both sides, such that $M\otimes_B^{\sf L} N\simeq A$ and $N\otimes_A^{\sf L} M \simeq B$ in $\sing(A\- A)$ and $\sing(B\- B)$ respectively.
\end{proposition}

If $A$ and $B$ are Gorenstein algebras then $M$ and $N$ could be replaced with  $B\- A$ bimodules (rather than a complexes of such), by taking syzygies. More classically, if $A$ and $B$ are self-injective then this is easily seen to coincide with the usual definition.

\begin{proof}
Note that ${}^\vee M=\RHom_{A^\op}(M,A)$ is naturally a dg $B\-A$-bimodule. Since ${}_AM$ is perfect there is a natural quasi-isomorphism $\-\otimes_A^{\sf L} M\to \RHom_{A}({}^\vee M,\-)$. It follows that there are natural isomorphisms
\[
\sHom_{A\-A}(A,\-)\xrightarrow{\-\otimes_A^{\sf L} M}\sHom_{A\-B}(M,\-\otimes_A^{\sf L} M)\xrightarrow{\cong}\sHom_{A\-B}(M,\RHom_{A}({}^\vee M,\-))
\]
\[
\hspace{75mm}\cong \sHom_{A\-A}(M\otimes_A^{\sf L} {}^\vee M,\-).
\]
So by the Yoneda lemma $M\otimes_B^{\sf L} {}^\vee M\simeq A$ in $\sing(A\- A)$. Similarly if we set $M^\vee=\RHom_B(M,B)$ then $M^\vee\otimes_A^{\sf L} M\simeq B$ in $\sing(B\- B)$. It follows that ${}^\vee M\simeq M^\vee$, by the usual argument for uniqueness of inverses.
\end{proof}

\section{Stable Invariance of the restricted Structure}
 
 A ring is called \emph{Iwanaga Gorenstein} if it is two-sided noetherian, and if it has finite injective dimension over itself as both a left and right module (Buchweitz calls these rings strongly Gorenstein \cite{Buc}). By a result of Zaks, this implies that the two injective dimensions $\idim {{}_AA}$ and $\idim A_A$ coincide \cite{Zaks}.

\begin{lemma}\label{Gorlemma}
Assume that $k$ is Gorenstein, and that $A$ is finite and projective over $k$ with finite self-injective dimension $d$ on both sides. Then $A^\ev$ has self-injective dimension at most $2d$ on either side, and
\[
\Ext^i_{A^\ev}(A,A)\to \sExt^i_{A^\ev}(A,A)
\]
is an isomorphism for $i> d$.
\end{lemma}

If would be enough to assume that $A$ is perfect over $k$ here (i.e. having a finite resolution by finitely generated projectives), but this would involve some additional dg algebra techniques. 

\begin{proof} 
Denote $(\-)^*=\RHom_k(\-,k)$. Since $k$ is Gorenstein $(M^*)^*\simeq M$ whenever $M$ is in $\dbcat(k)$ (e.g.~ \cite{Buc}).

If $M$ is a finitely generated $A^\ev$-module then 
\begin{align*}
\RHom_{A^\ev}(M,A^{\ev})& \simeq \RHom_{A^\ev}(M,\Hom(A^*,A))
\\
& \simeq \RHom_{A}(A^*\otimes^{\sf L}_AM,A)\\ &\simeq \RHom_{A}(A^*,\RHom_A(M,A))\\
&\simeq  \RHom_{A^\op}(\RHom_A(M,A)^*,A) .
\end{align*}
The last quasi-isomorphism follows from the Gorenstein duality mention above (which applies since $A$ and $\RHom_A(M,A)$ are in $\dbcat(k)$). The other quasi-isomorphisms are standard adjunctions. 

By inspection the last term in this string of isomorphisms has cohomology concentrated in degrees  at most $\idim {{}_AA}+\idim A_A = 2d$. This implies that $\idim A^\ev_{A^\ev}\leq 2d$. The same argument shows that $\idim {{}_{A^\ev}A^\ev}\leq 2d$.

Clearly $A^\ev$ is two-sided noetherian, therefore we've shown that it is Iwanaga Gorenstein. Because of this we can use the long exact sequence
\[
\cdots\to \sExt^{i-1}_{A^\ev}(A,A)\to {\rm Tor}^{A^\ev}_{-i}(A,A^\vee)\to \Ext^i_{A^\ev}(Y,X)\to\sExt^i_{A^\ev}(A,A)\to \cdots
\]
from \cite[theorem 6.2.5]{Buc}, where $A^\vee=\RHom_{A^\ev}(A,A^{\ev})$. Using the isomorphism above with $M=A$ we see that $
A^\vee\simeq \RHom_{A}(A^*,A)
$ 
has cohomology in degrees $d$ and below. This means that ${\rm Tor}^{A^\ev}_{-i}(A,A^\vee)=0$ for $i>d$, and the statement follows from the long exact sequence.
\end{proof}

\begin{theorem}\label{generalVersion}
Assume that $k$ is Gorenstein, and that $A$ and $B$ are finite and projective as $k$-modules, and that both have (left and right) self-injective dimension less than or equal to $d$. 
If $A$ and $B$ are connected by a stable equivalence of Morita type then there is an isomorphism
\[
 {\rm HH}^{>d}(A,A)\cong {\rm HH}^{>d}(B,B)
\]
which preserves the cup product, and (after shifting) the \underline{restricted} graded Lie-algebra structure.
\end{theorem}

\begin{proof}
This follows from observing the commutative diagram (\ref{StabilisationDiagram}) above. The two lower horizontal maps are isomorphisms by Theorem \ref{keytheorem}. The outside vertical maps are isomorphisms in degrees $d$ and above by lemma \ref{Gorlemma}. The upper horizontal maps preserve all the required structure by Theorem \ref{E2restrictedLie}. Combining these facts finishes the proof. 
\end{proof}

Our main interest, coming from the field of modular representation theory, is the case of self-injective algebras over fields. Specialising Theorem \ref{generalVersion} to this context yields Theorem \ref{stableinvariancetheorem} from the introduction, with one caveat to be explained in the next subsection.

\subsection{Transfer maps} 
\label{transfersection}
To finish the proof of Theorem 
\ref{stableinvariancetheorem} from the introduction, we need to explain why the isomorphism of Theorem \ref{generalVersion} is given by a transfer map.

 Transfer maps were introduced by Bouc  \cite{Bouc}  for Hochschild homology, and  defined for the Hochschild cohomology of symmetric algebras by Linckelmann  \cite{Linc}. They are the classical 
 way of comparing the Hochschild cohomology of two algebras 
 connected by a stable equivalence of Morita type. 
 
 Since the definition of stable equivalence of Morita type we use here is more general than usual, we need to explain what we mean by a transfer map.

If $M$ and $N$ are $A\-B$ and $B\-A$ bimodules respectively, we have a map
\[
\Ext_{A\-A}(A,A)\to \Ext_{B\-B}(N\otimes^{\sf L}_AM,N\otimes^{\sf L}_AM),\quad f\mapsto N\otimes^{\sf L}_A f\otimes^{\sf L}_AM.
\]
Assuming $M$ and $N$ are each of finite flat dimension on both sides, this descends to a map
\[
\sExt_{A\-A}(A,A)\to \sExt_{B\-B}(N\otimes^{\sf L}_AM,N\otimes^{\sf L}_AM).
\]
Suppose we are given an isomorphism $N\otimes^{\sf L}_AM\simeq B$ in $\csing(B\-B)$. The \emph{transfer map} is by definition the composition
\[
{}_N{\rm tr}_M : \sExt_{A\-A}(A,A)\to \sExt_{B\-B}(N\otimes^{\sf L}_AM,N\otimes^{\sf L}_AM)\xrightarrow{\cong} \sExt_{B\-B}(B,B).
\]
In the case of an algebra over a field, this agrees with the more concrete definition given by Linckelmann (where one usually uses $N=M^*$). In this context the transfer map is manifestly an algebra homomorphism, but it turns out \emph{not} to respect the restricted Lie algebra structure of $\sExt_{A\-A}(A,A)$ and  $\sExt_{B\-B}(B,B)$ in general, see the counter-example in \cite[section 6]{RyD}.

If moreover $M\otimes^{\sf L}_AN\simeq A$ then 
${}_N{\rm tr}_M$ and ${}_M{\rm tr}_N$ are inverse isomorphisms. Surprisingly, it \emph{is} true in this case that these isomorphisms respect the restricted Lie algebra structure; this follows from Theorem \ref{generalVersion} and the following.

\begin{proposition}\label{transferprop}
The isomorphism $\sExt_{A\-A}(A,A)\cong\sExt_{B\-B}(B,B)$ of Theorem \ref{keytheorem} coincides with ${}_N{\rm tr}_M$.
\end{proposition}
\begin{proof}[Sketch of proof.]
One can check that the following square commutes:
\[
\begin{tikzcd}
    \sExt_{T\-T}(T,T) \ar[r,"\underline{R}_B"]\ar[d,"\underline{R}
    _A"'] & \sExt_{B\-B}(B,B) \ar[d,"M \otimes_B^{\sf L} \-"] \\
    \sExt_{A\-A}(A,A) \ar[r, "\-\otimes_A^{\sf L} M"] \ar[ur, dashed]& \sExt_{A\-B}(M,M).
\end{tikzcd}
\]
The square consists of isomorphisms, which gives us the dashed arrow $\underline{R}_B(\underline{R}_A)^{-1}$. This is the isomorphism constructed in the proof of Theorem \ref{keytheorem}. Since $N\otimes_B^{\sf L} \- $ is the  inverse of $M \otimes_B^{\sf L} \-$, the dashed arrow is also given by $N \otimes_B^{\sf L} \-\otimes_B^{\sf L} M = {}_N{\rm tr}_M$.
\end{proof}

\begin{remark}[Generalisation to all of stable Hochschild cohomology]
The proof of our main theorem goes through stable Hochschild cohomology, but we do not make serious use of its structure, and we are not significantly interested in it as an invariant by itself (only because ordinary Hochschild cohomology is more computatable and more widely used). Nevertheless, we quickly point out now that one can generalize our main theorem to all of $\sHH(A,A)$ (so, into negative degrees).

Wang has studied the structure of stable Hochschild cohomology \cite{Wang}. He constructs a stable version of the Hochschild cochain complex and 
establishes an analogue of the Deligne conjecture for it. It follows from this and from Theorem \ref{E2restrictedLie} that $\susp \sHH(A,A)$ is naturally a 
restricted graded Lie algebra as well. By inspecting Wang's stable cochain 
complex, the arguments above can be used to show that this structure is invariant 
under stable equivalence of Morita type.

\end{remark}

\section{Applications in modular representation theory}
\subsection{The classification of tame symmetric algebras}
\label{tamesym}

Since we now know the restricted Lie structure of ${\rm HH}^1(A,A)$ to be an 
invariant under stable equivalence of Morita type, we can hope to use it in 
classification problems.

We contribute to improve the (as yet incomplete) classification of algebras of dihedral, semi-dihedral and quaternion type in characteristic $2$ up to stable equivalence of Morita type, after the 
work of Zhou and Zimmermann \cite{ZZ}, and  Taillefer \cite{Taillefer}. The Lie brackets of ${\rm HH}^1(A,A)$ for some of these algebras have been computed in \cite{Taillefer}. In \cite{ER} the authors prove the solvability of ${\rm HH}^1(A,A)$
when the Ext-quiver of  algebras  of dihedral, semi-dihedral and quaternion type  does not have loops. In \cite{RSS} the authors prove the solvability of  ${\rm HH}^1(A,A)$
for all algebras of 
dihedral, semi-dihedral and 
quaternion type excluding blocks of group algebras with Klein defect group. Thus, we 
 compute the restricted Lie structure for most of the remaining cases. The dihedral case has been 
completely classified in \cite {ZZ} and the two remaining families are considered in \cite{Taillefer}.  
Our calculations failed to distinguish any algebras which are currently not known 
to be in different stable Morita equivalence classes. This provides some evidence that most these classes are in fact stably equivalent.
In order to prevent others 
from having to repeat our calculations, we outline them below. 

In what follows we use
the notation of \cite{Taillefer}.

\begin{itemize}
    \item The algebras $SD(1\mathcal{A})^k_2(c,1)$ and $SD(1\mathcal{A})^k_1(c{'},
    1)$, for $c \neq c'$ non-zero, are 
    derived equivalent and their first  Hochschild cohomologies are isomorphic as
    Lie algebras. Using the same isomorphism
     of \cite[proposition 4.9]{Taillefer},  ${\rm HH}^1(SD(1\mathcal{A})^k_2(c,1))$ and ${\rm
    HH}^1(SD(1\mathcal{A})^k_1(c{'}, 1))$
    are isomorphic as restricted Lie algebras. 
    Therefore we are still not able to say whether or not these algebras are stably Morita equivalent.

    \item In the case of $Q(1\mathcal{A})^k_2(0,d)$ and
    $Q(1\mathcal{A})^k_2(0,d')$ for $k$ odd and $d\neq d{'}$ we have an isomorphism
    of restricted Lie algebras, but one needs to slightly change the isomorphism of \cite[Corollary 5.6]{Taillefer}. More precisely, the adjusted isomorphism 
    acts as: $\psi\mapsto\lambda\psi'$, $\chi\mapsto\tilde{\mu}\chi'$ and 
    $\theta_{-2} \mapsto \mu \theta_{-2}'$,
    where $\lambda^2=(\frac{d}{d'})^3$, $\tilde{\mu}=\frac{\lambda'}{d}$
    and $\mu=\frac{d}{d'\tilde{\mu}}$, in the notation of \cite{Taillefer}.
    \item In the case ${c, d} \neq {c', d'}$ with $cd \neq 0$ and $c{'}d{'} \neq 0$ we couldn't establish an isomorphism of restricted Lie algebras. Therefore we still do not know if $Q(1\mathcal{A})^k_2 (c, d)$ and $Q(1\mathcal{A})^k_2(c{'}, d{'})$ for  are stably equivalent of Morita type or not.
\end{itemize}

\subsection{Blocks of group algebras of defect \texorpdfstring{$2$}{}  and quantum complete intersections}
Let $p$ be an odd prime. Let $B$ be a block of a group algebra with defect group 
$D$ of order $p^2$ and let $C$ the  Brauer correspondent of $B$
(a block of $kN_G(D)$). Then by \cite{Rou} there is a 
stable equivalence of Morita type between $B$ and $C$. 
In addition, if $C$ has one simple module, then by \cite{KL} $B$ has
one simple module. If $B$ is a nilpotent block, then $B$ is Morita
equivalent to $kD$ and therefore $\rm{HH}^1(B,B)$ is a Jacobson-Witt algebra. Otherwise, 
following the same arguments of the proof of
\cite[Corollary 1.4]{BKL}, we have that $D$ is elementary 
abelian of rank 2 and 
a basic algebra of $C$ is a quantum
complete intersection $A$ of rank 2 with quantum parameter $q$
which is a root of unity. In \cite{BKL} 
the restricted structure of $\rm{HH}^1(A,A)$ is calculated. Then we have   

\begin{proposition}
Let $B$ a block of a group algebra $kG$ with defect group of order $p^2$ and with Brauer correspondent $C$. Assume that $C$ has one simple module and that $B$ is not nilpotent. Then $\rm{HH}^1(B)=\mathcal{H}\oplus \mathcal{S}$, where $\mathcal{H}$ is a maximal p-toral subalgebra and $\mathcal{S}$
is the derived subalgebra of $\rm{HH}^1(B)$.
\end{proposition}

Note that the validity of Brou\'e's abelian defect group conjecture 
would imply 
that $B$ is derived equivalent to $C$ and by a result of Roggenkamp and Zimmermann  \cite[Proposition 6.7.4]{ZR} we would have a Morita equivalence between $B$ and $C$. Kessar proves in  \cite{Kess} that the Morita equivalence holds for $p=3$.

\section{Applications in commutative algebra}
Suppose $A$ is a commutative Gorenstein ring essentially of finite type over a 
field $k$, with Krull dimension $d$. Then $A^\ev$ is  Gorenstein of 
Krull dimension $2d$, and   $\Ext_{A^\ev}^i(A,A^{\ev})=0$ for $i>d$ (by 
\cite[Theorem 2.1 (iv)]{AI} for example). Knowing this, the same proof as for 
theorem \ref{generalVersion} gives:

\begin{theorem}\label{commutativetheorem}
Assume that $A$ and $B$ are commutative Gorenstein rings essentially of finite type over a field $k$, both of Krull dimension less than or equal to $d$. 
If $A$ and $B$ are connected by a stable equivalence of Morita type then there is an isomorphism $
 {\rm HH}^{>d}(A,A)\cong {\rm HH}^{>d}(B,B)
 $ 
which preserves the cup product, and the restricted graded Lie-algebra structure.
\end{theorem}

\begin{example}\label{KnorrerEx}
A well-known family of stable equivalences of Morita type is given by Kn\"orrer's periodicity theorem  \cite{Knorrer}---but this is usually not phrased in these terms. 

Let $S=k[x_1,...,x_n]$ and let $f$ be a nonzero element of $S$. We set $A=S/(f)$ and $B=S[u,v]/(f+uv)$. Theses rings are Gorenstein of dimension $n-1$ and $n+1$ respectively. We consider $U=A[u]=B/(v)$ and $V=A[v]=B/(u)$ as $A\-B$ and $B\-A$ bimodules respectively. One can check that $U\otimes^{\sf L}_B V\simeq A$ in $\csing(A^\ev)$ and $V\otimes^{\sf L}_A U\simeq B$ in $\csing(B^\ev)$. Hence $A$ and $B$ are stably equivalent of Morita type (no assumptions on the field $k$ are needed). It follows from Theorem \ref{commutativetheorem} that  there is an isomorphism ${\rm HH}^{>n+1}(A,A)\cong {\rm HH}^{>n+1}(B,B)$ preserving the product and the restricted Lie structure.
\end{example}

For hypersurface rings
stable Hochschild cohomology is always $2$-periodic, and so ordinary Hochschild cohomology  becomes $2$-periodic above the dimension of the ring. In the case of Kn\"orrer periodicity, example \ref{KnorrerEx} says that ${\rm HH}^*(A,A)$ and $ {\rm HH}^*(B,B)$ have the same periodic part.

\appendix
\section{Well-definedness of the \texorpdfstring{$p$}{}-power operation}

Here we will explain briefly some ideas due to May and Cohen which can be used to exhibit the $p$-power operation on Hochschild cohomology.  These ideas are not well-known outside of the topology community, so it seems worthwhile to sketch them here. We have relegated the arguments to an appendix because we need to assume some knowledge of operads.  Our thanks are due to Victor Turchin for explaining Theorem \ref{Maythm} to us.

Our goal here is to finish the proof of Theorem \ref{E2restrictedLie} above: we need to show that the $p$-power operation is well-defined on the homology of a $B_\infty$-algebra. 
But let us point out that it is quite easy to see that the $p$-power operation on ${\rm HH}^1(A,A)$ is well-defined
; this appendix is only needed for classes in ${\rm HH}^{>1}(A,A)$.  Even though our main interest is ${\rm HH}^1(A,A)$, we include this appendix for the sake of completeness, and because these arguments are difficult to find in the literature.

As explained in section \ref{restrictedsubsection}, one can use the model $\coder(BA,BA)$ to exhibit the restricted structure on Hochschild cohomology. It is stated in \cite[Lemma 4.2]{ZII} that this approach yields a well-defined $p$-power operation, but the proof there is  incomplete. To establish well-definedness, one must show that for every odd cycle $x\in \coder^i(BA,BA)$ and every $y\in \coder^{i+1}(BA,BA)$ there is an element $w$ such that $(x+\partial(y))^p = x^p +\partial(w)$. It is not checked in \cite{ZII} that $w$ is a coderivation, and unfortunately the choice used there is not. It seems to be \emph{extremely} complicated to construct directly a suitable $w$, and hence a nonconstructive method is needed.

Here we assume knowledge of operads (specifically, $\Sigma$-split dg operads \cite{Hinich}). We work with the brace operad $\mathcal{B}$ \cite{GV}, which is an example of a $\Sigma$-split $E_2$-operad. It seems to be necessary to involve this machinery to establish that the $p$-power operation is well defined. The following two theorems together settle this matter.

The first is a result of May \cite{May};  it is a general template for building cohomology operations. Since  \cite{May} appeared just before operads were first defined, and since it was written in topological terms, it seems worthwhile to give an algebraic argument here.

\begin{theorem}[May]\label{Maythm}
Let $\O$ be a $\Sigma$-split dg operad, and let $A$ be a dg $\O$-algebra. Suppose that $\xi \in \O(n)^j$ is such that $\partial(\xi)$ becomes zero in $\O(n)\otimes_{\Sigma_n} ({\rm sign})$. Then, for $i$ odd, the operation
\[
{\rm H}^i(A)\to {\rm H}^{ni+j}(A),\quad a\mapsto \xi(a,...,a)
\]
is well-defined on cohomology classes.

Similarly, if $\partial(\xi)=0$ in $\O(n)\otimes_{\Sigma_n} ({\rm triv})$ then $\xi$ induces a well-defined operation on even cohomology classes.
\end{theorem}

The next theorem originates in topology, in the work Cohen \cite{Cohen}. The arguments of \cite{Cohen} are purely topological calculations concerning double loops spaces. The algebraic version here is essentially due to Turchin. 

\begin{theorem}[Cohen, Turchin]\label{Cohenthm}
Let $p$ be a prime and let $\B$ be the brace operad over $\mathbb{F}_p$. Denote by 
$\xi = (\-)\{\-\}\cdots\{\-\}$ the $p$-fold iterated brace operation (in Gerstenhaber's notation $\xi = ((\-\circ \-)\circ \cdots \circ \- ) \circ \-$).  Then $\partial(\xi)$ becomes zero in $\B(p)\otimes_{\Sigma_p} ({\rm sign})$. Hence $\xi$ induces a well-defined, $p$-power linear operation on the homology of any $\B$-algebra.
\end{theorem}

\begin{proof}[Proof of Theorem \ref{Maythm}.] 
We only prove the first statement, involving the sign representation. The statement about the trivial representation is proven the same way.

So assume that $\partial(\xi)=0$ in $\O(n)\otimes_{\Sigma_n} ({\rm sign})$, and take two cycles $a,a'\in A^i$ and $b\in A^{i-1}$ with $\partial(b)=a-a'$. Our goal is to show that $\xi(a,...,a)$ and $\xi(a',...,a')$ are cohomologous cycles.

Let $I$ be a one-dimensional vector space concentrated in degree $i$. We identify $a$ and $a'$ with the corresponding maps $I\to A$, and we consider the induced maps
\[
\begin{tikzcd}
 \O(n)\otimes_{\Sigma_n}(I^{\otimes n}) \ar[r, shift left=3pt, "\O(n)\otimes_{\Sigma_n}(a^{\otimes n})"] \ar[r, shift right=3pt, "\O(n)\otimes_{\Sigma_n}(a'^{\otimes n})"']  &[10mm] \O(n)\otimes_{\Sigma_n} A^{\otimes n} \ar[r,"\mu"] & A.
\end{tikzcd}
\]
Here $\mu$ is the given $\O$-action on $A$. Since $i$ is odd $I^{\otimes n}\cong({\rm sign})$ as left $\Sigma_n$-modules, so by  assumption $\xi$ represents a class in $\O(n)\otimes_{\Sigma_n}(I^{\otimes n})$. By construction
\[
\mu (\O(n)\otimes_{\Sigma_n}a^{\otimes n})(\xi) = (-1
)^{nj}\xi(a,...,a)\ \text{ and }\ \mu (\O(n)\otimes_{\Sigma_n}a'^{\otimes n})(\xi) = (-1
)^{nj}\xi(a',...,a').
\]
Hence it suffices to show that $\O(n)\otimes_{\Sigma_n}(a^{\otimes n})$ and $\O(n)\otimes_{\Sigma_n}(a'^{\otimes n})$ induce the same map in cohomology. 

Let $V$ be the linear span in $A$ of $a,a'$ and $b$. This is a subcomplex of $A$ (with dimension $3$ or less). Both $a,a'\colon I\rightrightarrows V$ are quasi-isomorphisms, and both are split by the retraction $r:V\to I$ with $r(b)=0$ and $r(a)=r(a')=$ the generator of $I$. Since $r$ is a quasi-isomorphism, so is $r^{\otimes n}: V^{\otimes n}\to I^{\otimes n}$. Since $\O$ is $\Sigma$-split, $\O(n)\otimes_{\Sigma_n}(r^{\otimes n})$ is also a quasi-isomorphism. Finally, since $\O(n)\otimes_{\Sigma_n}(a^{\otimes n})$ and $\O(n)\otimes_{\Sigma_n}(a'^{\otimes n})$ (or strictly speaking, the induced maps into $\O(n)\otimes_{\Sigma_n}(V^{\otimes n})$) are both split by the same quasi-isomorphism $\O(n)\otimes_{\Sigma_n}(r^{\otimes n})$, they must induce the same map in cohomology.
\end{proof}

In the next proof we will need certain permutations $\tau_i\in \Sigma_n$ given by
\[
\tau_i(j)=\begin{cases} j & \text{ if } j< n-i\\
n & \text{ if } n-i\leq j< n\\
n-i & \text{ if } j=n
\end{cases}
\]
And we will use the fact that ${\rm sign}(\tau_i)=(-1)^i$. We also need to record some relations which hold in the brace operad:
\begin{equation}\label{rel1}
\partial(\-\{\-\})= (\-\smile\-) - (\-\smile\-)\tau_1
\end{equation}
\begin{equation}\label{rel2}
(\-\smile \-)\{ \-\}= (\-\smile( \- \{\-\})) + ((\-\{\-\})\smile\-)\tau_1,
\end{equation}
These can be found in \cite{GV} (they both go back to Gerstenhaber's earlier work \cite{Gerst2}).

\begin{proof}[Proof of Theorem \ref{Cohenthm}.] 
We denote $\xi^{(n)}=(\-)\{\-\}\cdots\{\-\}$ with $n$ inputs. In particular $\xi^{(1)}$ is the identity and $\xi^{(p)}$ is $\xi$ itself. Note that $\xi^{(n)}$ has degree $1-n$. 

We will prove by induction on $n$ that in $\B(n)\otimes_{\Sigma_n} ({\rm sign})$ we have
\begin{equation}\label{indhyp}
\partial(\xi^{(n)})=\sum_{i=1}^{n-1}(-1)^{ni}\binom{n}{i}\xi^{(i)}\smile \xi^{(n-i)}.
\end{equation}
This formula was discovered by Turchin (with different signs), see \cite[Proposition 11.1]{TourtchineII}. For completeness we include a detailed proof.  The base case uses the relation (\ref{rel1}):
\[
\partial(\xi^{(2)})=(\-\smile\-) - (\-\smile\-)\tau  = 2 \xi^{(1)}\smile \xi^{(1)}.
\]
For the induction step we use the Leibniz rule 
\begin{equation*}
\partial(\xi^{(n+1)})  = \partial(\xi^{(2)}\circ_1\xi^{(n)}) = \partial(\xi^{(2)})\circ_1\xi^{(n)} - \xi^{(2)}\circ_1\partial(\xi^{(n)}).
\end{equation*}
We expand the first summand here using (\ref{rel1}):
\begin{align*}
\partial(\xi^{(2)})\circ_1\xi^{(n)} & = (\-\smile\-)\circ_1\xi^{(n)}-((\-\smile\-)\tau_1)\circ_1\xi^{(n)}\\
& = (\-\smile\-)\circ_1\xi^{(n)}-(\-\smile\-)\circ_2\xi^{(n)} \tau_n\nonumber\\
& = \xi^{(n)}\smile \xi^{(1)}+(-1)^{n+1}\xi^{(1)}\smile \xi^{(n)}.\nonumber
\end{align*}
And also the second summand:
\begin{align*}
 \xi^{(2)}\circ_1\partial(\xi^{(n)}) 
    & =  \sum_{i=1}^{n-1}(-1)^{ni}\binom{n}{i}\xi^{(2)}\circ_1\left[\xi^{(i)}\smile \xi^{(n-i)}\right]\\
    & = \sum_{i=1}^{n-1}(-1)^{ni}\binom{n}{i}\left[\xi^{(i+1)}\smile \xi^{(n-i)}\tau_{n-i}+(-1)^{i-1}\xi^{(i)}\smile \xi^{(n-i+1)}\right]\nonumber\\
    & = \sum_{i=1}^{n-1}\binom{n}{i}\left[(-1)^{ni+n-i}\xi^{(i+1)}\smile \xi^{(n-i)}+(-1)^{ni+i-1}\xi^{(i)}\smile \xi^{(n-i+1)}\right].\nonumber
\end{align*}
The first equality is the induction hypothesis; the second one uses (\ref{rel2}) precomposed with $\xi^{(i)}\otimes \xi^{(n-i)}$; the last equality just groups signs together. 
If we combine these two summands and group together coefficients we find
\begin{align*}
\partial(\xi^{(2)})\circ_1&\xi^{(n)} - \xi^{(2)}\circ_1\partial(\xi^{(n)})   \\
       & = -\sum_{i=1}^{n}\left[(-1)^{n(i-1)+n-i+1}\binom{n}{i-1}+(-1)^{ni+i-1}\binom{n}{i}\right]\xi^{(i)}\smile \xi^{(n+1-i)}\\
       & = \sum_{i=1}^{n}(-1)^{(n+1)i}\left[\binom{n}{i-1}+\binom{n}{i}\right]\xi^{(i)}\smile \xi^{(n+1-i)}\\
       & = \sum_{i=1}^{n}(-1)^{(n+1)i}\binom{n+1}{i}\xi^{(i)}\smile \xi^{(n+1-i)}.
\end{align*}
This completes the proof by induction that (\ref{indhyp}) holds. Since $\binom{p}{i}=0$ for $i=1,...,p-1$ 
this completes as well the proof that $\partial(\xi)=0$ in $\B(p)\otimes_{\Sigma_p}({\rm sign})$.

By May's theorem \ref{Maythm} above, $\xi$ induces a well-defined cohomology operation on any $\B$-algebra.
\end{proof}

\end{document}